\documentclass[12pt]{amsart}
\usepackage{amscd,amsmath,amsthm,amssymb}
\usepackage[left]{lineno}
\usepackage{color}

\usepackage{tikz}

\definecolor{verylight}{gray}{0.97}
\definecolor{light}{gray}{0.9}
\definecolor{medium}{gray}{0.85}
\definecolor{dark}{gray}{0.6}

 \def\KK{{\NZQ K}}
 
 \def\G{{\mathcal G}}

 \def\0b{{\mathbf 0}}

\def\reg{{\mathbf reg}}

\def\KK{{\mathbb K}}
\def\height{\operatorname{ht}}
\def\depth{\operatorname{depth}}
 \def\opn#1#2{\def#1{\operatorname{#2}}} 
 \opn\chara{char} \opn\length{\ell} \opn\pd{pd} \opn\rk{rk}
 \opn\projdim{proj\,dim} \opn\injdim{inj\,dim} \opn\rank{rank}
 \opn\depth{depth} \opn\grade{grade} \opn\height{height}
 \opn\embdim{emb\,dim} \opn\codim{codim}
 
 \opn\Tr{Tr} \opn\bigrank{big\,rank}
 \opn\superheight{superheight}\opn\lcm{lcm}
 \opn\trdeg{tr\,deg}
 \opn\reg{reg} \opn\lreg{lreg} \opn\ini{in} \opn\lpd{lpd}
 \opn\size{size} \opn\sdepth{sdepth}
 \opn\link{link}\opn\fdepth{fdepth}\opn\lex{lex}
 \opn\tr{tr}
 \opn\type{type}
 \opn\gap{gap}
 \opn\arithdeg{arith-deg}
 \opn\HS{HS}
 \opn\GL{GL}
 \opn\div{div} \opn\Div{Div} \opn\cl{cl} \opn\Cl{Cl}
 \opn\Spec{Spec} \opn\Supp{Supp} \opn\supp{supp} \opn\Sing{Sing}
 \opn\Ass{Ass} \opn\Min{Min}\opn\Mon{Mon}
 \opn\Ann{Ann} \opn\Rad{Rad} \opn\Soc{Soc}\opn\Deg{Deg}
 \opn\Im{Im} \opn\Ker{Ker} \opn\Coker{Coker} \opn\Am{Am}
 \opn\Hom{Hom} \opn\Tor{Tor} \opn\Ext{Ext} \opn\End{End}
 \opn\Aut{Aut} \opn\id{id}
 
 \opn\nat{nat}
 \opn\pff{pf}
 \opn\Pf{Pf} \opn\GL{GL} \opn\SL{SL} \opn\mod{mod} \opn\ord{ord}
 \opn\Gin{Gin} \opn\Hilb{Hilb}\opn\sort{sort}
 \opn\PF{PF}\opn\Ap{Ap}
 \opn\mult{mult}
 \opn\bight{bight}
 \opn\match{match}
\opn\St{St}

 \opn\aff{aff}
 \opn\relint{relint} \opn\st{st}
 \opn\lk{lk} \opn\cn{cn} \opn\core{core} \opn\vol{vol}  \opn\inp{inp} \opn\nilpot{nilpot}
 \opn\link{link} \opn\star{star}\opn\lex{lex}\opn\set{set}
 \opn\width{wd}
 \opn\Fr{F}
 \opn\QF{QF}
 \opn\G{G}
 \opn\type{type}\opn\res{res}
 \opn\conv{conv}
 \opn\Ind{Ind}
 \opn\gr{gr}
 
 \def\pot#1#2{#1[\kern-0.28ex[#2]\kern-0.28ex]}

 \opn\dirlim{\underrightarrow{\lim}}
 \opn\inivlim{\underleftarrow{\lim}}
 \def\Implies{\ifmmode\Longrightarrow \else
         \unskip${}\Longrightarrow{}$\ignorespaces\fi}
 \def\implies{\ifmmode\Rightarrow \else
         \unskip${}\Rightarrow{}$\ignorespaces\fi}
 \def\iff{\ifmmode\Longleftrightarrow \else
         \unskip${}\Longleftrightarrow{}$\ignorespaces\fi}

 \let\:=\colon
 \newtheorem{Theorem}{Theorem}[section]
 \newtheorem{Lemma}[Theorem]{Lemma}
 \newtheorem{Corollary}[Theorem]{Corollary}
 \newtheorem{Proposition}[Theorem]{Proposition}
 \newtheorem{Remark}[Theorem]{Remark}
 
 \newtheorem{Example}[Theorem]{Example}
 
 \newtheorem{Definition}[Theorem]{Definition}

 \newtheorem{Question}[Theorem]{Question}

\newcommand\calA{\mathcal{A}}

\newcommand\calC{\mathcal{C}}
\newcommand\calJ{\mathcal{J}}

 \let\epsilon\varepsilon
 \let\kappa=\varkappa
 \def\qed{\ifhmode\textqed\fi
       \ifmmode\ifinner\quad\qedsymbol\else\dispqed\fi\fi}
 \def\textqed{\unskip\nobreak\penalty50
        \hskip2em\hbox{}\nobreak\hfil\qedsymbol
        \parfillskip=0pt \finalhyphendemerits=0}
 \def\dispqed{\rlap{\qquad\qedsymbol}}
 
 \opn\dis{dis}
 \def\pnt{{\raise0.5mm\hbox{\large\bf.}}}
 
 \opn\Lex{Lex}

\begin{document}

\title{Castelnuovo-Mumford regularity of generalized binomial edge ideals of graphs}

\author{Dariush Kiani, Sara Saeedi Madani and Guangjun Zhu}



\address{Dariush Kiani, Department of Mathematics and Computer Science, Amirkabir University of Technology (Tehran Polytechnic), Tehran, Iran}
\email{dkiani@aut.ac.ir}

\address{Sara Saeedi Madani, Department of Mathematics and Computer Science, Amirkabir University of Technology (Tehran Polytechnic), Tehran, Iran and School of Mathematics, Institute for Research in Fundamental Sciences, Tehran, Iran}
\email{sarasaeedi@aut.ac.ir}

\address{Guangjun Zhu, School of Mathematical Sciences, Soochow University, Suzhou, Jiangsu, 215006, P.R. China}
\email{zhuguangjun@suda.edu.cn}

\thanks{The second author was  supported in part by a grant from IPM (No. 1404130019) and the third author is supported by the Natural Science Foundation of Jiangsu Province
(No. BK20221353) and the National Natural Science Foundation of China (No.12471246). }

\maketitle
\begin{abstract}
In this paper, we mainly study the Castelnuovo-Mumford regularity of the generalized binomial edge ideals of graphs. We show that this number can be any integer number from $2$ to $n-1$ where $n$ is the number of vertices in the underlying graph. We are able to show this, after giving some tight lower and upper bounds for the regularity of generalized binomial edge ideals of the join product of graphs. In particular, we characterize all generalized binomial edge ideals with the regularity equal to~$2$ as well as extremal Gorenstein ideals. For this purpose, we give a new combinatorial characterization for the class of $P_4$-free graphs. 
\end{abstract}




\maketitle

\section{Introduction}
\label{introduction}

Binomial edge ideals were defined independently in \cite{HHHKR} and \cite{O} and later became a popular topic of study in combinatorial commutative algebra. Let $G$ be a finite nonempty simple graph on the vertex set $V(G)$ and the edge set $E(G)$. Throughout this paper, all graphs are finite nonempty simple graphs, and if a graph has $n$ vertices, then we may use $[n]=\{1,\ldots,n\}$ to denote its set of vertices. Let $S=\KK[x_1,\ldots,x_n,y_1,\ldots,y_n]$ be the polynomial ring with $2n$ variables over a field $\KK$. Then the \emph{binomial edge ideal} of $G$, denoted by $J_G$, is defined as follows:
\[
J_G=(f_{ij}: 1\leq i<j\leq n, \{i,j\}\in E(G))
\]
where $f_{ij}=x_iy_j-x_jy_i$. If $G$ is a complete graph with $n$ vertices, then $J_G$ is indeed a determinantal ideal. Binomial edge ideals have been studied extensively by several authors, see for example \cite{BN, BMS, EHH, EHH1, EZ, HHHKR, HKS, KS2, MM, O, RS, RSK1, RSK2, RSK3, RSK4, SK, SK3, SZ}.

Later on, a generalization of binomial edge ideals, called the \emph{binomial edge ideal of a pair of graphs}, was defined in \cite{EHHQ}. Let $G_1$ and $G_2$ be two graphs on the vertex sets $[m]$ and $[n]$, respectively, with $m,n \geq 2$. Let $X=(x_{ij})$ be an $(m\times n)$-matrix of variables and let $S=\KK[x_{ij}:(i,j)\in [m]\times [n]]$ be the polynomial ring with variables~$x_{ij}$ where $K$ is a field. Then, let
\[
p_{e,f}=x_{it}x_{j\ell}-x_{i\ell}x_{jt}
\]
be a binomial in $S$ where $i<j$, $t<\ell$, $e=\{i,j\}\in E(G_1)$ and $f=\{t,\ell\}\in E(G_2)$. Then, the binomial edge ideal of the pair of graphs $(G_1,G_2)$, denoted by $J_{G_1,G_2}$, is defined as follows:
\[
J_{G_1,G_2}=(p_{e,f}: e\in E(G_1), f\in E(G_2)).
\]
If $G_1$ is just the complete graph on two vertices~$K_2$, then $J_{G_1,G_2}$ is just the classical binomial edge ideal of $G_2$. These very general types of ideals also have been considered by some authors, see for example \cite{EHHQ} and \cite{SK1}.

The case where $G_1$ is a complete graph $K_m$ for $m\geq 3$ has also been of considerable interest in several research papers in this area, as it is closer to the classical binomial edge ideal. Therefore, it is also known by the specific name of \emph{generalized binomial edge ideal}. In particular, the generalized binomial edge ideal of a graph $G$ on $[n]$ is denoted by $J_{K_m,G}$ in the polynomial ring $S$ which has also been considered by several authors, see for example \cite{IrCh, Ku1, Rauh, Zhu1, Zhu2, Zhu3}.

In the present paper, we are also interested in generalized binomial edge ideals of graphs and our focus is on the Castelnuovo-Mumford regularity (or regularity for short) of such ideals. We explain our main goals in more detail below.

Finding explicit formulas for the regularity of the generalized binomial edge ideal of a graph has been of interest to many authors and there are several results for certain classes of graphs, see for example~\cite{IrCh}, \cite{Ku1} and \cite{Zhu1}. On the other hand, we know from \cite[Theorem~3.6]{Ku1} that
\[
1\leq \reg S/J_{K_m,G}\leq n-1,
\]
for any connected graph $G$ with $n$ vertices and $m\geq 3$. A natural question is whether all integers in the above range are reached by some generalized binomial edge ideal. An analogous question was considered and answered in \cite{SK3} for $m=2$. More precisely, we may ask the following question about generalized binomial edge ideals.

\begin{Question}\label{any regularity}
 Given integers $n$ and $r$ with $1\leq r \leq n-1$, is there a connected graph $G$ on $n$ vertices for which $\reg S/J_{K_m,G}=r$ for some $m\geq 3$?
\end{Question}

If $r=1$, then the answer is positive. Even more, in \cite[Theorem~10]{SK1}, all pairs of graphs $G_1$ and $G_2$ were characterized for which $\reg(S/J_{G_1,G_2})=1$. Based on that characterization, we have $\reg(S/J_{K_m,G})=1$ if and only if $G=K_2$. For $r=n-1$, several cases are given in \cite{Ku1}, but so far there is no complete characterization. In this paper, we answer Question~\ref{any regularity} for the remaining cases, namely, $2\leq r\leq n-2$.

Motivated by the aforementioned explicit combinatorial characterization for the smallest possible value, one may ask for a characterization in each possible value for the regularity. More precisely, we ask the following general question.

\begin{Question}\label{regularity characterization}
Let $n\geq 4$. Suppose the answer to Question~\ref{any regularity} is positive for an integer $r$ with $2\leq r \leq n-1$. Then is there an explicit and combinatorial characterization for graphs $G$ such that $\reg S/J_{K_m,G}=r$ for some $m\geq 3$?
\end{Question}

As the next step after the case $r=1$, it is reasonable to consider the case $r=2$ which is one of our main goals in this paper. For the case $m=2$, such a characterization has been given in \cite{SK3}.



\medskip
This paper is organized as follows. In Section~\ref{join}, we answer Question~\ref{any regularity} for $3\leq r\leq n-2$. Indeed, we do much more, by considering a nice construction from graph theory, called the join product of graphs. We prove some tight lower and upper bounds for the regularity of the generalized binomial edge ideal of the join product of two graphs and discuss certain cases for which the exact value of the regularity can be computed. Then, as an application of the results of this section, we settle Question~\ref{any regularity} for $3\leq r\leq n-2$. In Section~\ref{P4 free section}, we focus on the interesting class of $P_4$-free graphs, also known as cographs, i.e. graphs with no induced subgraph isomorphic to the path $P_4$. Our approach in that section is totally combinatorial and we give a new combinatorial characterization for $P_4$-free graphs. This characterization, which is by its own interesting in graph theoretical point of view, will be then crucial to prove our results in Section~\ref{section reg 2}. Then, in Section~\ref{section reg 2}, we give a positive answer to Question~\ref{regularity characterization} for $r=2$ which, in particular, together with our results in Section~\ref{join}, puts an end to Question~\ref{any regularity}. In Section~\ref{extremal}, we consider generalized binomial edge ideals of regularity~$2$ and characterize those with the Cohen-Macaulay property which leads us to a characterization of all extremal Gorenstein generalized binomial edge ideals of graphs. In Section~\ref{Further remarks}, we conclude the paper with some related questions and remarks for future work.

\section{Regularity of the generalized binomial edge ideal of the join product}
\label{join}

In this section, we first consider the join product of two graphs $G_1$ and $G_2$ and give tight lower and upper bounds for the regularity of the generalized binomial edge ideal of the join product. We discuss the equality in the bounds in several cases. Then, as an application of our results, we give a positive answer to Question~\ref{any regularity} for $3\leq r\leq n-2$. First, we recall some notation and facts that we will need later in the section.

Let $G$  be a simple graph with the vertex set $V(G)$ and the edge set $E(G)$. For $W\subseteq V(G)$, we denote by $G[W]$ the \emph{induced subgraph} of $G$ on the vertex set $W$, i.e., for $u,v \in W$, $\{u,v\} \in E(G[W])$ if and only if $\{u,v\}\in E(G)$, and we also denote by $G\setminus W$, the induced subgraph
of  $G$  on $V(G) \setminus W$. For  $v\in V(G)$, $G\setminus v$ denotes the induced subgraph of $G$ on the vertex set $V (G)\setminus \{v\}$.  Let $c(G)$ denote the number
of connected components of $G$. A vertex $v$ is called a \emph{cut vertex} of $G$ if $c(G)< c(G\setminus v)$.
For  $T\subset V(G)$, we set $\overline{T}=V(G)\setminus T$. If  $v$ is a cut vertex of the graph $G[\overline{T} \cup \{v\}]$ for any $v\in T$,
then we say that $T$  has the \emph{cut vertex property}. Set \[
\calC(G)=\{\emptyset\}\cup\{T: T \text{\ has the cut vertex property}\}
\]
and
\[
\overline{\calC}(G)=\calC(G)\setminus \{\emptyset\}.
\]
Let   $\widetilde{G}$ denote the complete graph on $V(G)$, and let $c(T)$ denote the number of connected components of $G[\overline{T}]$.


Let $G$ be a simple  graph with the  vertex set $[n]$.  For each  $T\subseteq V(G)$,  we can introduce the following ideal in $S=\KK[x_{ij}: 1\le i\le m, 1\le j\le n]$
\[
    P_T(K_m, G)=(x_{ij}: (i,j)\in [m]\times T)+J_{K_m,\widetilde{G}_1}+\cdots+J_{K_m,\widetilde{G}_{c(T)}},
\]
where $G_1,\ldots, G_{c(T)}$ are the connected components of $G[\overline{T}]$.
It was shown in \cite[Theorem~7]{Rauh} that $T\in \calC(G)$ if and only if $P_T(K_m,G)$ is a minimal prime ideal of $J_{K_m,G}$, and hence
\[
J_{K_m,G}=\bigcap_{T\in \calC(G)}P_T(K_m,G).
\]

 \medskip
\begin{Definition}\label{join-def}
Let  $r\ge 2$ be an integer, and let $G_1,\ldots,G_r$ be  graphs on   pairwise  disjoint vertex sets $V_1,\ldots, V_r$, respectively. The join product of  $G_1,\ldots,G_r$, denoted by  $G_1*G_2*\cdots*G_r$,  is a graph on the vertex set $V_1\cup \cdots \cup V_r$,  whose  edge set is
\[
E(G_1)\cup\cdots \cup E(G_r)\cup \{\{x, y\}: x\in V_i, y\in  V_{j}\text {\ with\ }i\ne j\}.
\]
\end{Definition}


Let $V$ be a set and $\calA_1,\ldots,\calA_t$ be $t$ collections of subsets of $V$. The \emph{join} of $\calA_1,\ldots,\calA_t$, denoted by $\bigcirc_{i=1}^t\calA_i$,
was introduced in \cite{KS1} as
\[
\{A_1\cup A_2\cup\cdots\cup A_t: A_i\in \calA_i \text{\ for\ } i=1,\ldots,t\}.
\]
By convention, $\bigcirc_{i=1}^t\calA_i=\emptyset$ if there exists some $i$ for which $\calA_i=\emptyset$.

\medskip
The next proposition describes the minimal prime ideals of the generalized binomial edge ideals of the join product of two graphs.
If $G$ is a graph with connected components $G_1,\ldots,G_r$, then we write $G$ as  $G=\bigsqcup_{i=1}^r G_i$.

\begin{Proposition}\label{both disconnected 1}
	\cite[Proposition~4.14]{KS1}
	Let $G_1=\bigsqcup_{i=1}^r G_{1i}$ and $G_2=\bigsqcup_{i=1}^s G_{2i}$ be two graphs on disjoint sets of vertices
	$V_1=\bigcup_{i=1}^r V_{1i}$ and $V_2=\bigcup_{i=1}^s V_{2i}$,
	respectively, where $r,s\geq 2$. Then
	$$\calC(G_1*G_2)=\{\emptyset\}\cup \big{(}(\bigcirc_{i=1}^{r}\calC(G_{1i}))\circ V_2 \big{)}\cup \big{(}(\bigcirc_{i=1}^{s}\calC(G_{2i}))\circ V_1 \big{)}.$$
\end{Proposition}

For each induced subgraph, the following proposition provides a useful formula.
\begin{Proposition}\label{induced}
	\cite[Proposition~8]{SK1}
	Let $G$ be a graph and $H$ be its induced subgraph. Then $\reg (S/J_{K_m,H})\leq \reg (S/J_{K_m,G})$.
\end{Proposition}

We also benefit from the following two propositions, one of which gives the exact value of the regularity for the generalized binomial edge ideal of a complete graph. They partially answer Question~\ref{any regularity}.

\begin{Proposition}\label{both complete}
	\cite[Proposition~3.3]{Ku1}
	Let $K_n$ be the complete graph with $n$ vertices. Then $\reg (S/J_{K_m,K_n})= \min \{m-1,n-1\}$.
\end{Proposition}

\begin{Theorem}\label{Matsuda-Murai}
	\cite[Theorem~3.7]{Ku1}
	Let $G$ be a connected graph with $n$ vertices. If $m\geq n$, then
	\[
	\reg (S/J_{K_m,G})= n-1.
	\]
\end{Theorem}

We would like to remark that if, more generally, $G$ is a graph with $n$ vertices with $m\geq n$ consisting of $r$ connected components, then
\[
\reg (S/J_{K_m,G})= n-r\leq n-1.
\]

\begin{Lemma}  {\em \cite[Corollary 18.7]{P}}
	\label{exact}
	Let $0\longrightarrow M\longrightarrow N\longrightarrow P\longrightarrow 0$ be a short exact
	sequence of finitely generated graded S-modules. Then we have
	\[
\reg(M)\leq \max\{\reg(N), \reg(P)+1\}.
\]
\end{Lemma}

\medskip
Now we are ready to prove the main theorem of this section.

\begin{Theorem}\label{reg-join}
	Let $G_1$ and $G_2$ be two graphs on disjoint sets of vertices, with $n_1$ and $n_2$ vertices, respectively and let $G=G_1*G_2$. If $m<n_1+n_2$, then
	\[\max\{m-1,\reg(S/J_{K_m,G_i})|i\in[2]\}\leq \reg(S/J_{K_m,G})\leq \max\{m,\reg(S/J_{K_m,G_i})|i\in [2]\}.\]
	In other words,
	\[
	\reg(S/J_{K_m,G})\in \{\reg(S/J_{K_m,G_1}), \reg(S/J_{K_m,G_2}), m-1, m\}.
	\]
\end{Theorem}

\begin{proof} If $G$ is a complete graph, then both $G_1$ and $G_2$ are also complete graphs.  In this case, both of the lower and upper bounds follow from Proposition~\ref{both complete}.
So we can assume that $G$ is not complete, i.e., at least one of $G_1$ and $G_2$ is not complete.

We can assume that  $n_1\leq n_2$ and $n=n_1+n_2$.  Since both $G_1$ and $G_2$ are the induced subgraphs of $G$, by Proposition~\ref{induced}, we have $\reg(S/J_{K_m,G})\geq \reg(S/J_{K_m,G_i})$ for each $i\in [2]$. It follows that
\begin{equation}\label{lower bound}
\reg(S/J_{K_m,G})\geq \max\{\reg(S/J_{K_m,G_1}),\reg(S/J_{K_m,G_2}),m-1\}
\end{equation}
by \cite[Corollary~3.8]{Ku1}.

To check another inequality,  we first consider the case where both of $G_1$ and $G_2$ are disconnected graphs.  Let $G_1=\bigsqcup_{i=1}^r G_{1i}$ and $G_2=\bigsqcup_{i=1}^s G_{2i}$ with disjoint sets of vertices
	$V_1=\bigcup_{i=1}^r V_{1i}$ and $V_2=\bigcup_{i=1}^s V_{2i}$,
	respectively, where $r,s\geq 2$. Applying Proposition~\ref{both disconnected 1}, we can decompose $\calJ_{K_m,G}$ as
            $\calJ_{K_m,G}=Q\cap Q'$, where
	\begin{equation}
	Q=\bigcap_{\substack{
			T\in \mathcal{C}(G) \\
			V_1\subseteq T
	}}
	P_T(K_m,G)~~,~~Q'=\bigcap_{\substack{
			T\in \mathcal{C}(G) \\
			V_1\nsubseteq T
	}}
	P_T(K_m,G).
	\nonumber
	\end{equation}
	Furthermore, it follows from Proposition \ref{both disconnected 1} that
	\begin{align*}
	Q&=(x_{ij}:(i,j)\in [m]\times V_1)+\bigcap_{\substack{
			T\in \mathcal{C}(G) \\
			V_1\subseteq T
	}}
	P_{T\setminus V_1}(K_m,G_2)\\&=(x_{ij}:(i,j)\in [m]\times V_1)+J_{K_m,G_2},
	\nonumber
 \end{align*}
	and
	\begin{align}
	Q'&=P_{\emptyset}(K_m,G)\cap\big{(}\bigcap_{\substack{
			\emptyset\neq T\in \mathcal{C}(G) \\
			V_1\nsubseteq T
	}}
	P_T(K_m,G)\big{)}\\
    &= P_{\emptyset}(K_m,G)\cap \big{(}(x_{ij}:(i,j)\in [m]\times V_2)+\bigcap_{\substack{
			T\in \mathcal{C}(G) \\
			V_2\subseteq T
	}}
	P_{T\setminus V_2}(K_m,G_1)\big{)}\nonumber\\
   &=J_{K_m,K_n}\cap \big{(}(x_{ij}:(i,j)\in [m]\times V_2)+J_{K_m,G_1}\big{)}.
	\nonumber
	\end{align}
Hence
	\[
	Q+Q'=(x_{ij}:(i,j)\in [m]\times V_1)+J_{K_m,K_{n_2}}.
	\]
	Indeed, it is obvious that $(x_{ij}:(i,j)\in [m]\times V_1)+J_{K_m,K_{n_2}}\subseteq Q+Q'$, and the other inclusion
follows, since $Q$ and $J_{K_m,K_n}$ are contained in $(x_{ij}: i\in [m], j\in V_1)+J_{K_m,K_{n_2}}$.
 Consequently, we have
	\begin{align*}
	\reg(S/Q)&=\reg(S/J_{K_m,G_2}),\\
\reg(S/(Q+Q'))&=\reg(S/J_{K_m,K_{n_2}})=\min\{m-1,n_2-1\}
	\end{align*}
 by Proposition~\ref{both complete}. Applying Lemma  \ref{exact} to the following short exact sequence
	\[
	0\rightarrow S/J_{K_m,G}\rightarrow S/Q\oplus S/Q'\rightarrow S/(Q+Q')\rightarrow 0,
	\]
	we obtain
	\[
	\reg(S/J_{K_m,G})\leq \max\{\reg(S/Q),\reg(S/Q'),\reg(S/(Q+Q'))+1\}.
	\]
Hence
	\begin{equation}\label{lower bound equality 1}
	\reg(S/J_{K_m,G})\leq \max\{\reg(S/J_{K_m,G_2}),\reg(S/Q'),
	\min\{m,n_2\}\}.
	\end{equation}
	To compute $\reg(S/Q')$, we consider the exact sequence
	\[
	0\rightarrow S/Q'\rightarrow S/J_{K_m,K_n}\oplus S/(I+J_{K_m,G_1})\rightarrow S/(I+J_{K_m,K_{n_1}})\rightarrow 0,
	\]
	where $I=(x_{ij}:(i,j)\in [m]\times V_2)$.  Applying Lemma  \ref{exact} again, we get
	\[
	\reg(S/Q')\leq
	\max\{\reg(S/J_{K_m,K_n}),\reg(S/J_{K_m,G_1}),
	\reg(S/J_{K_m,K_{n_1}})+1\}.
	\]
Note that $m<n$, by Proposition~\ref{both complete}, we have
	\[
	\reg(S/Q')\leq\max\{m-1,\reg(S/J_{K_m,G_1}),\min\{m,n_1\}\}.
	\]
We distinguish the following cases:\\

(I) If $m\leq n_1$, then
	\[
	\reg(S/Q')\leq\max\{\reg(S/J_{K_m,G_1}), m\}.
	\]

(II)  If $m>n_1$, then
	\begin{equation} \label{lower bound equality 2}
	\reg(S/Q')\leq\max\{\reg(S/J_{K_m,G_1}), m-1\}.
	\end{equation}
By the above two cases and (\ref{lower bound equality 1}), we get
	\[
	\reg(S/J_{K_m,G})\leq \mathrm{max}\{\mathrm{reg}(S/J_{K_m,G_1}),\mathrm{reg}(S/J_{K_m,G_2}),m\}.
	\]
In the following, we can assume that either $G_1$ or $G_2$ is connected.  We can obtain two disconnected graphs $G_1'$ and $G_2'$
by adding isolated vertices $v$ and $w$ to $G_1$ and $G_2$, respectively. Let $G'=G_1'*G_2'$ and let $S'$ be the polynomial ring over the field $K$ obtained by adding $2m$ variables, corresponding to the vertices $v$ and $w$, to the polynomial ring $S$. Then $G$ is an induced subgraph of $G'$ and it is clear that the homogeneous binomial generating sets of $J_{K_m,G_i}$ and $J_{K_m,G_i'}$ coincide for each $i=1,2$.

Note that $m<n_1+n_2+2$, it follows from  Proposition~\ref{induced} and the above discussion that
\begin{align*}
	\reg(S/J_{K_m,G})&\le \reg(S'/J_{K_m,G'})\leq \max\{\reg(S'/J_{K_m,G_1'}),\reg(S'/J_{K_m,G_2'}),m\}\\
&=\max\{\reg(S/J_{K_m,G_1}),\reg(S/J_{K_m,G_2}),m\}.
	\end{align*}
	This completes the proof.
\end{proof}

\medskip
Next we discuss certain conditions  for attaining one of the bounds given in Theorem~\ref{reg-join}.  The next corollary  concerns the lower bound.
\begin{Corollary}\label{equal to m-1}
	Let $G_1$ and $G_2$ be two graphs on disjoint sets of vertices, with $n_1$ and $n_2$ vertices, respectively, where $n_1\leq n_2$. Let $G=G_1*G_2$ and  $3\leq m<n_1+n_2$. If  one of the following conditions holds:
	\begin{enumerate}
	\item $G_1$ and $G_2$ are disconnected and $n_2 < m$;
	\item $n_2+1 < m$.
	\end{enumerate}
		Then
		\[
		\reg(S/J_{K_m,G})=m-1.
		\]
	\end{Corollary}
	
	\begin{proof}
		We use the notation from Theorem~\ref{reg-join}. If $G_1$ and $G_2$ are disconnected and $n_2<m$, then by Eqs. (\ref{lower bound equality 1}) and (\ref{lower bound equality 2}) in the proof of Theorem~\ref{reg-join}, we get
\[
\reg(S/J_{K_m,G})\leq \max\{\reg(S/J_{K_m,G_1}),\reg(S/J_{K_m,G_2}),m-1\}.
		\]
		On the other hand, by Theorem~\ref{Matsuda-Murai} we have $\mathrm{reg}(S/J_{K_m,G_i})\leq n_i-1$ for each $i=1,2$.
		Hence
		\[
		\reg(S/J_{K_m,G})\leq m-1.
		\]
		It follows that from  the lower bound given in Theorem~\ref{reg-join} that
		\[
		\reg(S/J_{K_m,G})= m-1.
		\]
		
		If $n_2+1<m$, then the desired result can be obtained by similar arguments as in above. In this case, if $G_1$ or $G_2$ is connected, then it is enough to replace $G_1$ or $G_2$ by $G'_1$ or $G'_2$, as well as $n_1$ or $n_2$ by $n_1+1$ or $n_2+1$, respectively.
	\end{proof}


\begin{Remark}\label{lower and upper bound}
 \begin{enumerate}
{\em \item {\bf (Lower Bound)} Note that Corollary~\ref{equal to m-1} gives an exact formula for the regularity of the generalized binomial edge ideals for many non-trivial and unknown cases. More explicitly, let $G_1$ and $G_2$ be any arbitrary graphs with  $t$ and $t+j$ vertices, respectively,  where $t\ge 1$, $j\geq 0$ and $2t+j\geq 3$. Then for each $m=t+j+2, t+j+3, \ldots, 2t+j-1$, we have $\mathrm{reg}(S/J_{K_m,G_1*G_2})=m-1$, while $\mathrm{reg}(S/J_{K_m,G_1})\leq t-1$ and $\mathrm{reg}(S/J_{K_m,G_2})\leq t+j-1$ by Theorem~\ref{Matsuda-Murai}, and hence the maximum is given by $m-1$.

\item {\bf (Upper Bound)} If $G=K_{1,m}$ is a star graph with $m+1$ vertices where $m\geq 2$, then $\mathrm{reg}(S/J_{K_m,G})=m$ by \cite[Lemma~3.8]{Zhu3}. While $K_{1,m}$ is the join product of $G_1=K_1$ and $G_2=\bigsqcup_{i=1}^m K_1$, so $J_{K_m,G_i}=0$ for each $i=1,2$, and hence it is clear that  $\reg(S/J_{K_m,G_i})=0$ for any $i=1,2$. This shows  that the upper bound  in Theorem~\ref{reg-join} can be achieved in certain cases, and in particular that the maximum is given by $m$.
     }
	\end{enumerate}
\end{Remark}

The following consequence shows that the lower and upper bounds in Theorem 2.6 can also coincide, and this gives a precise formula for the regularity of the generalized binomial edge ideals of certain graphs.
\begin{Corollary}\label{equality of lower and upper bounds}
Let $G_1$ and $G_2$ be two graphs on disjoint sets of vertices, with $n_1$ and $n_2$ vertices, respectively. Let $G=G_1*G_2$ and  $m<n_1+n_2$. If $\mathrm{reg}(S/J_{K_m,G_i})\geq m$ for some $i=1,2$, then
	\[
      \reg(S/J_{K_m,G})=\max\{\reg(S/J_{K_m,G_1}), \reg(S/J_{K_m,G_2})\}.
	\]
\end{Corollary}
By the assumption and by Theorem~\ref{Matsuda-Murai}, the graph $G_i$, whose generalized binomial edge ideal has regularity at least $m$, consists of at least $m+1$ vertices,  but no restriction is needed for the other graph  involved in the join product. As an example of Corollary~\ref{equality of lower and upper bounds}, we can choose $G_1=K_{1,m}$ with $m\geq 2$, and $G_2$ is any graph.  The other candidate for $G_1$, instead of $K_{1,m}$, is given in the next example when $m=3$ which will also be useful for our characterization in the next section.

\begin{Example}\label{triangle-whisker}
Let $G_1=K_1$, $G_2=K_1\sqcup K_2$ and $G=G_1*G_2$. Then $\reg(S/J_{K_3,G})=3$, see also \cite[page~364]{Ku1}. As we will use this graph later, we will denote it by $B$.
\end{Example}

We end this section by giving a positive answer to Question~\ref{any regularity} for $3\leq r\leq n-2$, using our construction in this section  on the join product of graphs and the regularity results.

\begin{Corollary}\label{answer question 1}
Let   $n$ and $r$ be  two integers such that $3\leq r \leq n-2$. Then there exists a connected graph $G$ with $n$ vertices and $\reg(S/J_{K_m,G})=r$ for some $m\geq 3$.
\end{Corollary}

\begin{proof} Let  $3\leq m\leq r$, $n=n_1+n_2$, and let $G_1$ and $G_2$ be  two graphs with $n_1$ and $n_2$ vertices, respectively,  such that $\reg(S/J_{K_m,G_1})=r$ and $\reg(S/J_{K_m,G_2})\leq r$. Also, let $G=G_1*G_2$. It follows from Theorem~\ref{reg-join} that $\reg(S/J_{K_m,G})=r$. In particular, we can choose $G_1=P_{r+1}$ and $G_2=K_{n-r-1}^c$, where $K_{n-r-1}^c$ is the complement of the complete graph $K_{n-r-1}$. Then $\reg(S/J_{K_m,K_{n-r-1}^c})=0$ and $\reg(S/J_{K_m,P_{r+1}})=r$ by \cite[Corollary~3.6]{IrCh}.
\end{proof}

\section{$P_4$-free graphs}\label{P4 free section}

As usual, we denote the path graph with $k$ vertices by $P_k$. In this section, we provide a graph theoretical characterization of $P_4$-free graphs which will be useful in answering Questions~\ref{any regularity} and \ref{regularity characterization} for $r=2$ in the next section. This class of graphs has been of great interest in graph theory and there are several interesting characterizations for such graphs in the literature. Therefore, they are known by some other names, such as cographs or a comparability graph of a series-parallel partial order. Nevertheless, to the best of our knowledge, the characterization given in this section is a new one in the literature.

\medskip
First, we will review some graphical notions that we will need for our purposes in this section. Let $G$ and $H$ be two graphs. Then $G$ is said to be an $H$-free graph if it has no induced subgraph isomorphic to $H$. Namely, $P_4$-free graphs have no induced subgraphs isomorphic to $P_4$.

Let $v$ be a vertex of $G$. The set of all neighbors of $v$ in $G$ is denoted by $N_G(v)$. Moreover, we set $N_G[v]=N_G(v)\cup \{v\}$. More generally, if $T\subseteq V(G)$, then we set $N_G(T)=\cup_{v\in T} N_G(v)$ and $N_G[T]=N_G(T)\cup T$.

A subset $T$ of $V(G)$ is a \emph{dominating set} of $G$ if for any $v\in V(G)\setminus T$, we have $v\in N_G(T)$. The subset $T$ of $V(G)$ is called a \emph{connected dominating set} of $G$ if it is a dominating set of $G$ such that the induced subgraph $G[T]$ is connected. A subset of $V(G)$ that is minimal with respect to this property is called a \emph{minimal connected dominating set} of $G$. A \emph{minimum connected dominating set} of $G$ is then a connected dominating set of $G$ of the minimum size.

The next theorem from \cite{CS} gives a nice description for $P_k$-free graphs for $k\geq 4$.

\begin{Theorem} \label{Dominating}
	\cite[Theorem~4]{CS}
	Let $G$ be a connected $P_k$-free graph where $k\geq 4$. If $T\subseteq V(G)$ is a minimum connected dominating set of $G$, then $G[T]$ is a $P_{k-2}$-free graph or isomorphic to $P_{k-2}$.
\end{Theorem}

We are now ready to present the main theorem of this section. We would like to point out that the proof of the following theorem is mainly based on the proof of \cite[Theorem~3.2]{SK3}.

\begin{Theorem}\label{P_4-free}
	A graph is $P_4$-free if and only if each of its connected components with at least two vertices is the join product of two $P_4$-free graphs.
\end{Theorem}

\begin{proof}
 First, note that a join product $G_1*G_2$ is $P_4$-free if and only if both of $G_1$ and $G_2$ are $P_4$-free. On the other hand, it is clear that each of the connected components of a graph $G$ is a $P_4$-free graph if and only if $G$ is $P_4$-free. Therefore, it remains to prove that if $G$ is a connected $P_4$-free graph with at least two vertices, then the there exist two graphs $G_1$ and $G_2$ such that $G=G_1*G_2$.

 Since $G$ is connected, it has a minimum connected dominating set. Let $T$ be a minimum connected dominating set of $G$. Then it follows from Theorem~\ref{Dominating} that $G[T]$ is $P_2$-free or isomorphic to $P_2$. If $G[T]$ is $P_2$-free, then it consists of a single vertex, say $u$. Therefore, $G=G[V(G)\setminus \{u\}]*K_1$. If $G[T]$ is isomorphic to $P_2$, then it consists of an edge $e=\{v,w\}$. We set
 \[
 A=N_G(v)\setminus N_G[w],
 \]
 \[
 B=N_G(w)\setminus N_G[v],
 \]
 \[
 C=N_G(w)\cap N_G(v).
 \]

 If $A=\emptyset$, then we have $G=G[V(G)\setminus \{w\}]*K_1$. If $B=\emptyset$, then we have $G=G[V(G)\setminus \{v\}]*K_1$. Now, suppose that $A\neq \emptyset$ and $B\neq \emptyset$.

 Observe that any vertex in $A$ and any vertex in $B$ are adjacent in $G$. Indeed, if there is a vertex $v'\in A$ and a vertex $w'\in B$ that are not adjacent in $G$, then the subgraph induced on the vertex set $\{v', v, w, w'\}$ is isomorphic to $P_4$, which is a contradiction.

 Let $x\in C$. Then, we have $A\subseteq N_G(x)$ or $B\subseteq N_G(x)$. Otherwise, there is $v'\in A$ and $w'\in B$ which are not adjacent to $x$. Then, the induced subgraph of $G$ on the vertex set $\{x, v, v', w'\}$ is isomorphic to $P_4$, which is a contradiction. We set
 \[
 C_1=\{x\in C: A\subseteq N_G(x), B\nsubseteq N_G(x)\},
 \]
 \[
 C_2=\{x\in C: A\nsubseteq N_G(x), B\subseteq N_G(x)\},
 \]
 \[
 C_3=\{x\in C: A\subseteq N_G(x), B\subseteq N_G(x)\}.
 \]
 Then it follows that $C$ is partitioned by $C_1$, $C_2$ and $C_3$.


 Now, observe that any vertex in $C_1$ and any vertex in $C_3$ are adjacent in $G$. To see this, assume that there exist $x\in C_1$ and $z\in C_3$ which are not adjacent in $G$. Since $B\nsubseteq N_G(x)$, there is a vertex $w'\in B$ which is not adjacent to $x$. Then, the subgraph induced on the vertex set $\{w', z, v, x\}$ is isomorphic to $P_4$, which is a contradiction. Similarly, it follows that every vertex in $C_2$ and every vertex in $C_3$ are adjacent in $G$.

 Next, observe that any vertex in $C_1$ and any vertex in $C_2$ are adjacent in $G$. To see this, assume that there exist $x\in C_1$ and $y\in C_2$ which are not adjacent in $G$. Then, from the definitions of $C_1$ and $C_2$, it follows that $B\nsubseteq N_G(x)$ and $A\nsubseteq N_G(y)$. Thus, there exist vertices $w'\in B$ and $v'\in A$ which are not adjacent to the vertices $x$ and $y$, respectively. Therefore, the subgraph induced on the vertex set $\{x, v', w', y\}$ is isomorphic to $P_4$, a contradiction.

 Finally, let $G_1=G[\{v\}\cup B\cup C_1\cup C_3]$ and $G_2=G[\{w\}\cup A\cup C_2]$. Then, it follows that $G=G_1*G_2$, as desired.
\end{proof}

\section{Generalized binomial edge ideals of regularity~2}\label{section reg 2}


In this section we conclude Question~\ref{any regularity} by considering the case $r=2$. Furthermore, in this section, we also give a positive answer to Question~\ref{regularity characterization} for $r=2$. For this purpose, we divide the discussion of our characterization into two main parts, namely Lemma~\ref{reg 2-1} and Lemma~\ref{reg 2-2}.

\begin{Lemma} \label{reg 2-1}
 Let $G$ be a graph with $n$ vertices and  no isolated vertices, and let $m\geq n\geq 2$ be an integer. Then the following statements are equivalent:
 \begin{enumerate}
 	\item $\reg(S/J_{K_m,G})=2$.
 	\item $G=P_3$ or $K_3$ or $K_2\sqcup K_2$.
 \end{enumerate}
\end{Lemma}

\begin{proof}
First suppose that $G$ is a connected graph. Then it follows from  Theorem~\ref{Matsuda-Murai} that $\reg(S/J_{K_m,G})=2$ if and only if $n=3$. The only connected graphs with $3$ vertices are $P_3$ and $K_3$. Therefore $G=P_3$ or $K_3$ and in both cases, we have $\reg(S/J_{K_m,G})=2$ by \cite[Corollary~3.6]{IrCh} and Proposition~\ref{both complete}. Next, suppose that $G$ is disconnected with connected components $G_1,\ldots, G_c$ where $c\geq 2$ and for each $i$, $G_i$ is a graph with $n_i$ vertices. Then
\[
\reg(S/J_{K_m,G})=\sum_{i=1}^c \reg(S/J_{K_m,G_i})
\]     	
and $m>n_i$ for each $i$. Again, using Theorem~\ref{Matsuda-Murai}, it follows that
\[
\reg(S/J_{K_m,G})=\sum_{i=1}^c (n_i-1)=n-c.
\]
Hence $\reg(S/J_{K_m,G})=2$ if and only if $n=c+2$ if and only if $G=K_2\sqcup K_2$, since $G$  has no  isolated vertices. This completes the proof.
\end{proof}

Next, for the second part of this section, we need to establish some notation for complete multipartite graphs. Let  $r\geq 2$ be  an integer,  let $K_{t_1,\ldots,t_r}$ be a complete $r$-partite graph with the partition of its vertex set $V=V_1\sqcup \cdots\sqcup V_r$ and $|V_i|=t_i\geq 1$ for each $i=1,\ldots,r$
 such that each vertex of $V_i$ is adjacent to each vertex of $V_j$ for $1\leq i\neq j\leq r$ and no two vertices in each $V_i$ are adjacent in $K_{t_1,\ldots,t_r}$. We can assume that $t_i\leq t_{i+1}$ for each $i=1,\ldots,r-1$. It is clear from the definition that any complete $r$-partite graph can be seen as a join product, i.e, $K_{t_1}^c*\cdots *K_{t_r}^c$, where each $K_{t_i}^c$ is the complement of the complete  graph $K_{t_i}$.




\begin{Lemma} \label{reg 2-2}
Let $G$ be a graph with $n$ vertices and with no isolated vertices, and let
$n> m\geq 3$ be an integer. Then $\reg(S/J_{K_m,G})=2$ if and only if $m=3$ and $G$ is either
\begin{enumerate}
	\item $K_{t_1,\ldots,t_r}$ with $t_i\leq 2$ for each $i\in [r]$, or
	\item $K_2\sqcup K_2$.
\end{enumerate}
\end{Lemma}

\begin{proof}
If $G=K_{t_1,\ldots,t_r}$ with each $t_i\leq 2$, then $\reg(S/J_{K_3,G})=2$ by \cite[Theorem~3.9]{Zhu3}. If $G=K_2\sqcup K_2$, then by Proposition \ref{both complete}, we have
\[
\reg(S/J_{K_3,G})=2 \reg(S/J_{K_3,K_2})=2.
\]

 To prove the converse,  we first assume that $G$ is connected. Then $\reg(S/J_{K_m,G})\geq m-1$ by \cite[Corollary~3.8]{Ku1}. Since $\reg(S/J_{K_m,G})=2$, it follows that $m\leq 3$, which implies that $m=3$ by the assumption.
 By \cite[Corollary~3.6]{IrCh}, $\reg(S/J_{K_3,P_4})=3$, which forces  $G$  to be $P_4$-free by Proposition~\ref{induced}. Thus,  by Theorem~\ref{P_4-free},  $G=G_1*G_2$ where $G_1$ and $G_2$ are $P_4$-free graphs.
 Note that since both  $G_1$ and $G_2$ are induced subgraphs of $G$, we have $\reg(S/J_{K_3,G_i})\leq 2$. It follows from Example~\ref{triangle-whisker} and Proposition~\ref{induced} that
 $G$ is a $B$-free graph. In particular, both of $G_1$ and $G_2$ are $(K_1\sqcup K_2)$-free, since they have at least one vertex.
 On the other hand,  by \cite[Lemma~3.8]{Zhu3}, we have $\reg(S/J_{K_3,K_{1,3}})=3$, so $G$ is also $K_{1,3}$-free, and $G_1$ and $G_2$ are $K_3^c$-free.

 We proceed the proof by induction on the number of pairs of nonadjacent vertices in the graph. We denote this number for a graph $H$ by $b(H)$. If $b(G)=0$, then 
 $G$ is complete and isomorphic to $K_{t_1,\ldots,t_n}$ with $t_1=\cdots=t_n=1$. Now suppose that $b(G)>0$, then either $b(G_1)>0$ or $b(G_2)>0$. We may assume that $b(G_2)>0$. Then, there exist two nonadjacent vertices $v$ and $w$ in $V(G_2)$. Let $H$ be the induced subgraph of $G_2$ on the vertex set $V(G_2)\setminus \{v,w\}$. Then, it follows that the vertices $v$ and $w$ are adjacent to all the vertices of $H$, because $G_2$ is $K_{3}^c$-free as well as $(K_1\sqcup K_2)$-free. This implies that $G_2=H*K_2^c$ where $K_2^c$ is indeed isomorphic to the induced subgraph of $G_2$ on $\{v,w\}$. Therefore, we have $G=\hat{G_1}*K_2^c$ where $\hat{G_1}$ is the subgraph of $G$ induced on the vertex set $V(G_1)\cup V(H)$. Since $\hat{G_1}$ is an induced subgraph of $G$, it follows from Proposition~\ref{induced} that $\reg(S/J_{K_3,\hat{G_1}})\leq 2$. Now, we distinguish the following cases:

  (i) If $\reg(S/J_{K_3,\hat{G_1}})=0$, then $\hat{G_1}$ has no edges. Therefore, $V(H)=\emptyset$ and $G_1$ consists of at most two isolated vertices, since it is $K_3^c$-free. Thus, $G_1=K_1$ or $K_2^c$, and hence $G=K_{1,2}$ or $K_{2,2}$.

  (ii) If $\reg(S/J_{K_3,\hat{G_1}})=1$, then it follows from \cite[Theorem~10]{SK1} that $\hat{G_1}=K_2$ which implies that $G=K_{1,1,2}$.

  (iii) If $\reg(S/J_{K_3,\hat{G_1}})=2$, then $\hat{G_1}$ is connected. Indeed, if $\hat{G_1}$ is disconnected, then $V(H)=\emptyset$ and $G_1$ is disconnected. Since $G_1$ is $K_3^c$-free and $(K_1\sqcup K_2)$-free, it follows that $\hat{G_1}=G_1=K_2^c$, and hence $\reg(S/J_{K_3,\hat{G_1}})=0$, a contradiction. It is clear that $b(\hat{G_1})<b(G)$. Therefore, by induction hypothesis, $\hat{G_1}=K_{t_1,\ldots,t_s}$ for some $s$ with $t_i\leq 2$ for each $i\in [s]$. Then, we have $G=K_{t_1,\ldots,t_s,2}$.

Finally, we assume that $G$ is a disconnected graph with connected components $H_1,\ldots, H_c$ where $c\geq 2$ and for each $i$, $H_i$ is a graph with $n_i$ vertices. So
\[
\reg(S/J_{K_3,G})=\sum_{i=1}^c \reg(S/J_{K_3,H_i}).
\]     	
This implies that $c=2$ and $\reg(S/J_{K_3,H_i})=1$ for each $i=1,2$, since $G$ has no isolated vertices. So $H_1=K_2$ and $H_2=K_2$,  which  completes the proof.
\end{proof}

Finally, we conclude with the main theorem of this section, demonstrating that Question~\ref{regularity characterization} has a positive answer when $r=2$ through application of Lemmas~\ref{reg 2-1} and~\ref{reg 2-2}. This theorem moreover completes the answer to Question~\ref{any regularity}.

\begin{Theorem} \label{reg 2-main}
	Let $m\geq 3$ and $G$ be a graph with $n\geq 3$ vertices and no isolated vertices. Then $\reg(S/J_{K_m,G})=2$ if and only if one of the following conditions holds:
	\begin{enumerate}
		\item $G$ is $P_3$, $K_3$ or $K_2\sqcup K_2$;
		\item $m=3$ and $G=K_{t_1,\ldots,t_r}$ with $t_i\leq 2$ for each $i\in [r]$.
	\end{enumerate}
\end{Theorem}

\section{Extremal Gorenstein generalized binomial edge ideals}\label{extremal}

After known all generalized binomial edge ideals of regularity~2, in this section we will search for extremal Gorenstein generalized binomial edge ideals. First we recall the definition. It is well-known that if $I$ is a graded ideal in $R$, a polynomial ring  over a field, such that  $R/I$ is Gorenstein, then $I$ can never have linear resolution unless it is a principal ideal. Then, roughly speaking, if the minimal graded free resolution of $I$ is as linear as possible, then $I$ is called \emph{extremal Gorenstein}. More precisely, in the case of generalized binomial edge ideal $J_{K_m,G}$ of a graph $G$ which is generated in degree~$2$, $J_{K_m,G}$ is called extremal Gorenstein if $\reg (S/J_{K_m,G})=2$ and $S/J_{K_m,G}$ is Gorenstein. In the case of $m=2$, the desired characterization was discussed in \cite{SK3} with an improvement in \cite{RS}.

First, we characterize all graphs $G$ for which $S/J_{K_m,G}$ is Cohen-Macaulay with regularity $2$.

\begin{Theorem}\label{CM reg 2}
	Let $m\geq 3$ and $G$ be a graph with $n\geq 3$ vertices and no isolated vertices. Then the following statements are equivalent:
	\begin{enumerate}
		\item $S/J_{K_m,G}$ is Cohen-Macaulay and $\reg (S/J_{K_m,G})=2$;
		\item $G=K_3$ or $G=K_2\sqcup K_2$.
	\end{enumerate}
\end{Theorem}

\begin{proof}
We only need to determine which of the cases listed in Theorem~\ref{reg 2-main} gives a Cohen-Macaulay generalized binomial edge ideal. We distinguish three cases:

(i) It follows from \cite[Corollary~3.5]{IrCh} that if $G$ is a connected block graph, then $S/J_{K_m,G}$ is Cohen-Macaulay if and only if $G$ is a complete graph.
Therefore, $G$ can not be $P_3$, but $S/J_{K_m,K_3}$ is Cohen--Macaulay.

(ii) If $G=K_2\sqcup K_2$, then  $S/J_{K_m,K_2\sqcup K_2}\cong (S_1/J_{K_m,K_2})\otimes_{\KK} (S_2/J_{K_m,K_2})$, where $S_1$ and $S_2$ are the corresponding polynomial rings. Thus $S/J_{K_m,K_2\sqcup K_2}$  is Cohen-Macaulay, since $J_{K_m, K_2}$ is just the binomial edge ideal $J_{K_m}$, and hence $S/J_{K_m, K_2}$ is Cohen-Macaulay. So $G$ can also be $K_2\sqcup K_2$.

(iii) Suppose $G=K_{t_1,\ldots,t_r}$ is a  complete $r$-partite graph, where each $t_i\leq 2$. If each  $t_i=1$, then $G$ is a complete graph $K_r$ with $r\geq 3$. So $S/J_{K_m,G}$ is Cohen-Macaulay by \cite[Remark 3.2]{Zhu3}.
Now  we  assume that $t_r=2$. By  \cite[Proposition~3.5]{Zhu3}, we have
\[
\dim(S/J_{K_m,G})=\max \{m+n-1,mt_r\},
\]
where $n=\sum\limits_{i=1}^{r}t_i$. In particular, we have
\begin{equation}\label{dim}
\dim(S/J_{K_3,G})=\max \{n+2,6\}.
\end{equation}

If $n=3$, then $t_1=1$ and $t_2=2$. In this case, $G=P_3$, which was discussed earlier that $S/J_{K_m,G}$ is not Cohen--Macaulay. In the following,  we assume that $n\geq 4$. By Eq. (\ref{dim}), we have
\[
\dim (S/J_{K_3,G})=n+2.
\]
On the other hand, it has been shown in \cite[Theorem~3.6]{Zhu3} that
\[
\depth( S/J_{K_m,G})=m+\min\{t_i: t_i>1, i\in [r]\}.
\]
It follows that  $\depth (S/J_{K_3,G})=5$, i.e.  $\dim (S/J_{K_3,G})\neq \depth (S/J_{K_3,G})$, since $n\geq 4$. So $S/J_{K_3,G}$ is not Cohen-Macaulay.             	
\end{proof}

Now, we determine extremal Gorenstein generalized binomial edge ideals which are indeed rare.

\begin{Corollary}\label{Gor reg 3}
	Let $m\geq 3$ and  $G$ be a graph with no isolated vertices. Then $J_{K_m,G}$ is extremal Gorenstein if and only if $m=3$ and $G=K_3$.
\end{Corollary}

\begin{proof}
	It follows from \cite[Corollary~8.9]{BV} that $S/J_{K_m,K_2}$ is not Gorenstein. Now, we assume that $n\geq 3$. We only need to investigate about the Gorenstein property of the cases appearing in the statement of Theorem~\ref{CM reg 2}, namely $S/J_{K_m,K_3}$ and $S/J_{K_m,K_2\sqcup K_2}$ with $m\geq 3$.
	
	Again by \cite[Corollary~8.9]{BV}, we have that $S/J_{K_m,K_3}$ is Gorenstein if and only if $m=3$.
	Note that $S/J_{K_m,K_2}$ is a determinantal ring and thus Cohen-Macaulay, so its last Betti number is not $1$. On the other hand, the minimal graded free resolution of $S/J_{K_m,K_2\sqcup K_2}$ can be obtained by the tensor product of the minimal graded free resolutions of $S_1/J_{K_m,K_2}$ and $S_2/J_{K_m,K_2}$, where $S_1$ and $S_2$ are the corresponding polynomial rings with disjoint sets of variables. Thus the last Betti number of $S/J_{K_m,K_2\sqcup K_2}$ is also not  $1$, and hence $S/J_{K_m,K_2\sqcup K_2}$ is not Gorenstein. This completes the proof.
\end{proof}

\section{Further remarks and questions}\label{Further remarks}

Beside Question~\ref{regularity characterization} which is still widely open, in this section we would also like to make more specific questions and remark which are interesting to be done further.

We gave some examples of graphs $G$ with $n\geq 4$ vertices for which $\reg (S/J_{K_3,G})=3$. Then, as a further step, it can be reasonable to consider the following question.

\begin{Question}\label{K_3,reg=3}
	Is there an explicit combinatorial characterization for all graphs $G$ with $\reg S/J_{K_3,G}=3$?
\end{Question}

On the other hand, in particular in Theorem~\ref{reg 2-main} we characterized all graphs $G$ with $n\geq 3$ vertices and $\reg (S/J_{K_3,G})=2$. This temps us to ask the following question.

\begin{Question}\label{K_m,reg=m-1}
	Is there an explicit combinatorial characterization for all graphs $G$ with $\reg (S/J_{K_4,G})=3$? More generally, what about the case $\reg (S/J_{K_m,G})=m-1$ for any $m\geq 4$?
\end{Question}

In Section~\ref{join}, we discussed several cases for which the inequalities given in Theorem~\ref{reg-join} become equality. In particular, a special attention was given to the cases $\reg(S/J_{K_m,G_1*G_2})=m-1$ and $\reg(S/J_{K_m,G_1*G_2})=m$. It will be interesting to have a full characterization of such graphs $G_1$ and $G_2$. More precisely, we pose the following two questions. The answer to the first one also gives a partial answer to the general case of Question~\ref{K_m,reg=m-1}.

\begin{Question}\label{join-m,m-1}
 Is there an explicit combinatorial characterization for all graphs $G_1$ and $G_2$ with $n_1$ and $n_2$ vertices, respectively and $2\leq m<n_1+n_2$ such that $\reg(S/J_{K_m,G_1*G_2})=m-1$.
\end{Question}

 \begin{Question}\label{join-m,m}
 	Is there an explicit combinatorial characterization for all graphs $G_1$ and $G_2$ with $n_1$ and $n_2$ vertices, respectively and $2\leq m<n_1+n_2$ such that $\reg(S/J_{K_m,G_1*G_2})=m$.
 \end{Question}




\begin{thebibliography}{99}

\bibitem{BN} A. Banerjee, L. Nunez-Betancourt, {\em Graph Connectivity and Binomial edge ideals}, Proceed. of the Am. Math. Soc., Vol. 145 (2), 2017, 487-499.


\bibitem{BMS} D. Bolognini, A. Macchia, F. Strazzanti, {\em Binomial edge ideals of bipartite graphs}, Eur. J. Comb., 70 (2018), 1-25. 



\bibitem{BV} W. Bruns, U. Vetter, {\em Determinantal rings}, Springer, (1988).





\bibitem{CS} E. Camby, O. Schaudt, {\em A new characterization of $P_k$-free graphs}, Algorithmica, 75(1), (2016), 205-217.




\bibitem{EHH} V. Ene, J. Herzog, T. Hibi, {\em Cohen-Macaulay binomial edge ideals}, Nagoya Math. J. 204 (2011), 57-68.


\bibitem{EHH1} V. Ene, J. Herzog, T. Hibi, {\em Koszul binomial edge ideals}, Bridging Algebra, Geometry, and Topology, Volume 96 of the series Springer Proceedings in Mathematics and Statistics, (2014), 125-136.

\bibitem{EHHQ} V. Ene, J. Herzog, T. Hibi, A. A. Qureshi, {\em The binomial edge ideal of a pair of graphs}, Nagoya Math. J., 213 (2014), 105--125.

\bibitem{EZ} V. Ene, A. Zarojanu, {\em On the regularity of binomial edge ideals}, Math. Nachr. 288,  No. 1 (2015), 19-24.




\bibitem{HHHKR} J. Herzog, T. Hibi, F. Hreinsd{\'o}ttir, T. Kahle, J. Rauh, {\em Binomial edge ideals and conditional independence statements},
Adv. Appl. Math., 45 (2010), 317-333.

\bibitem{HE} M. Hochster and J. A. Eagon, A class of perfect determinantal ideals, Bull. Amer. Math. Soc. 76 (1970), 1026-1029.


\bibitem{HKS} J. Herzog, D. Kiani and S. Saeedi Madani, {\em The linear strand of determinantal facet ideals}, Michigan. Math. J., 66 (2017), 107-123.

\bibitem{IrCh}  F. Chaudhry, R. Irfan, {\em On the generalized binomial edge ideals of generalized block graphs}, Math. Reports, 22(72) 3-4 (2020), 381-394.


\bibitem{KS} D. Kiani, S. Saeedi Madani, {\em Binomial edge ideals with pure resolutions}, Collect. Math., 65 (2014), 331-340.


\bibitem{KS1} D. Kiani, S. Saeedi Madani, {\em Some Cohen-Macaulay and unmixed binomial edge ideals}, Comm. Algebra, 43 (2015), 5434-5453.


\bibitem{KS2} D. Kiani, S. Saeedi Madani, {\em The Castelnuovo-Mumford regularity of binomial edge ideals}, J. Combin. Theory Ser. A., 139 (2016), 80-86.

\bibitem{Ku1} A. Kumar, {\em Regularity bound of generalized binomial edge ideal of graphs}, J. Algebra, 546 (2020), 357--369.



\bibitem{MM} K. Matsuda, S. Murai, {\em Regularity bounds for binomial edge ideals}, J. Commutative Algebra, 5(1) (2013), 141-149.


\bibitem{O} M. Ohtani, {\em Graphs and ideals generated by some 2-minors}, Comm. Algebra, 39 (2011), 905-917.


\bibitem{P} I. Peeva, {\em Graded syzygies}, Springer, (2010).






\bibitem{Rauh} J. Rauh, {\em Generalized binomial edge ideals}, Adv. in Appl. Math., 50(3) (2013) 409--414.


\bibitem{RS} G. Rinaldo, R. Sarkar, {\em Level and pseudo--Gorenstein binomial edge ideals}, J. Algebra, 632 (2023), 363--383.


\bibitem{RSK1} M. Rouzbahani Malayeri, S. Saeedi Madani, D. Kiani, {\em Regularity of binomial edge ideals of chordal graphs}, Collect. Math., 72 (2021), 411--422.


\bibitem{RSK2} M. Rouzbahani Malayeri, S. Saeedi Madani, D. Kiani, {\em A proof for a conjecture on the regularity of binomial edge ideals}, J. Combinatorial Theory, Series A, 180 (3) (2021), 105432.


\bibitem{RSK3} M. Rouzbahani Malayeri, S. Saeedi Madani, D. Kiani, {\em Binomial edge ideals of small depth}, J. Algebra, 572 (2021), 231--244.


\bibitem{RSK4} M. Rouzbahani Malayeri, S. Saeedi Madani, D. Kiani, {\em On the depth of binomial edge ideals of graphs}, J. Algebraic Combin., 55 (3) (2022), 827--846.


\bibitem{SK} S. Saeedi Madani, D. Kiani, {\em Binomial edge ideals of graphs}, Electron. J. Comb., 19(2) (2012), $\sharp$ P44.


\bibitem{SK1} S. Saeedi Madani, D. Kiani, {\em On the binomial edge ideal of a pair of graphs}, Electron. J. Comb., 20(1) (2013), $\sharp$ P48.


\bibitem{SK3} S. Saeedi Madani, D. Kiani, {\em Binomial edge ideals of regularity 3}, J. Algebra, 515 (2018), 157--172.


\bibitem{SZ} P. Schenzel, S. Zafar, {\em Algebraic properties of the binomial edge ideal of a complete bipartite graph},  An. St. Univ. Ovidius Constanta, Ser. Mat. 22(2) (2014), 217-237.






\bibitem{Zhu1} Y. Shen, G. Zhu, {\em Powers of generalized binomial edge ideals of path graphs}, Int. J Algebr Comput.,  34(5) (2024) 779-806.

\bibitem{Zhu2} Y. Shen, G. Zhu, {\em Generalized binomial edge ideals of bipartite graphs}, J. Algebra, 662 (2025) 167-213.


\bibitem{Zhu3} Y. Shen, G. Zhu, {\em Generalized binomial edge ideals of complete $r$-partite graphs}, arXiv:2312.11807.


\end{thebibliography}
\end{document}